\documentclass[final,leqno]{siamltex}

\usepackage{amsmath,amsfonts} 
\usepackage{amssymb}
\usepackage{color}
\usepackage{enumerate}

\newcommand{\wt     }[1] {\widetilde{#1}}

\def\XXint#1#2#3{{\setbox0=\hbox{$#1{#2#3}{\int}$}
     \vcenter{\hbox{$#2#3$}}\kern-.5\wd0}}

\newcommand{\nifra  }[2] {\leavevmode\kern.1em\raise.5ex
  \hbox{\the\scriptfont0 #1}\kern-.1em/\kern-.15em\lower.25ex
  \hbox{\the\scriptfont0 #2}}
\newcommand{\mnifra  }[2] {\leavevmode\kern.1em\raise.5ex
  \hbox{\the\scriptfont0 $#1$}\kern-.1em/\kern-.15em\lower.5ex
  \hbox{\the\scriptfont0 $#2$}}

\newcommand{\z      }    {\frac{1}{2}}
\newcommand{\eps    }    {\varepsilon}

\newcommand{\dimx}{d}
\newcommand{\xb}{\overline x}

\newcommand{\telem}{\kappa}			
\newcommand{\teleme}{\telem_e}		
\newcommand{\belemx}{\partial \telem_{*}}	

\newcommand{\intk}{\int_\telem}	
\newcommand{\intdk}{\int_{\partial\telem}}	
\newcommand{\intdkx}{\int_{\belemx}} 
\newcommand{\no}{\nu}

\newcommand{\h}{h}

\newcommand{\triagn}{\mathcal T^n_\h}
\newcommand{\triagnn}[1]{\mathcal T^{#1}_\h}
\newcommand{\triag}{\mathcal T_\h}

\newcommand{\sumnk}{\sum_{n,\telem}}

\newcommand{\dx}{\, \mathrm{d}x}
\newcommand{\ds}{\, \mathrm{d}s}
\newcommand{\dxb}{\, \mathrm{d}\overline x}
\newcommand{\Rn}{\mathbb{R}^{\dimx}_n}
\newcommand{\Rnn}[1]{\mathbb{R}^{\dimx}_{#1}}

\newcommand{\dxr}{\, \mathrm{d}}

\newcommand{\normgen}[1]{|\!|#1|\!|}	
\newcommand{\seminorm}[1]{|#1|}		
\newcommand{\projh}{\pi_\h}			

\newcommand{\lpnorm}[3]{\Vert #1\Vert_{L^{#2}(#3)}}	

\newcommand{\poldeg}{q}
\newcommand{\altpoldeg}{p}
\newcommand{\Vsp}{V_{\h,\poldeg}}
\newcommand{\Vspn}{V_{\h,\poldeg}^n}

\newcommand{\polsp}{\mathcal P}
\newcommand{\polp}{\polsp^\poldeg}

\newcommand{\Lp}[1]{{L^{#1}}}
\newcommand{\Ltwo}{\Lp{2}}
\newcommand{\Linf}{\Lp{\infty}}
\newcommand{\Wmp}[2]{W^{#1,#2}} 
\newcommand{\Hm}[1]{H^#1} 

\newcommand{\uu}{u}			
\newcommand{\uh}{\uu_\h}		
\newcommand{\vv}{v}			
\newcommand{\vvp}{v^+}			
\newcommand{\vp}{\phi}
\newcommand{\uhp}{\uu_\h^+}
\newcommand{\uhm}{\uu_\h^-}
\newcommand{\uhk}{\uu_\h^{(\telem)}}
\newcommand{\uhkn}{\uu_\h^{(\teleme)}}

\newcommand{\fhat}{\hat f}

\newcommand{\fst}{\wt f}
\newcommand{\q}{F}				
\newcommand{\qst}{\wt \q}		
\newcommand{\epsvisc}{\eps}		
\newcommand{\entr}{\eta}
\newcommand{\entrp}{\entr'}
\newcommand{\entrpp}{\entr''}

\newcommand{\eentr}{e_{{\entr},\h}}

\newcommand{\semilin}[2]{\mathcal N(#1,#2)}
\newcommand{\semilinDG}[2]{\mathcal N^{DG}(#1,#2)}
\newcommand{\semilinS}[2]{\mathcal N^{SC}(#1,#2)}

\newcommand{\viscexp}{\beta}
\newcommand{\viscexpc}{\frac{\beta}{3}}

\newtheorem{thrm}{Theorem}[section]
\newtheorem{lem}{Lemma}[section]

\newtheorem{prop}{Proposition}[section]
\newtheorem{cor}{Corollary}[section]

\newcommand{\f}{f}
\newcommand{\g}{g}
\newcommand{\hh}{h}

\title{On the Convergence of Space-Time Discontinuous Galerkin Schemes for Scalar Conservation Laws}

\author{Georg May\footnotemark[1] \and Mohammad Zakerzadeh\footnotemark[1]  }

\begin{document}

\maketitle

\renewcommand{\thefootnote}{\fnsymbol{footnote}}
\footnotetext[1]{Aachen Institute for Advanced Study in Computational Engineering Science, RWTH Aachen, 52062 Aachen, Germany 
        (\texttt{\{may,zakerzadeh\}@aices.rwth-aachen.de}).}
\renewcommand{\thefootnote}{\arabic{footnote}}

\begin{abstract}
We prove convergence of a class of space-time discontinuous Galerkin schemes for scalar hyperbolic conservation laws. Convergence to the unique entropy solution is shown for all orders of polynomial approximation, provided strictly monotone flux functions and  a suitable shock-capturing operator are used. The main improvement, compared to previously  published results of similar scope, is that  no streamline-diffusion stabilization is used. This  is the way discontinuous Galerkin schemes were originally proposed, and are most often used in practice. 
\end{abstract}



\pagestyle{myheadings}
\thispagestyle{plain}
\markboth{\uppercase{G. May AND M. Zakerzdeh}}{\uppercase{Convergence of DG Schemes for Scalar Conservation Laws}}

\section{Introduction}

In the present paper we consider space-time discontinuous Galerkin (DG) schemes for the scalar conservation law
%
\begin{align}
	\label{eq:manProb1}
	\uu_t +\sum_{i=1}^d f_i(\uu)_{x_i} &= 0 ,  &&\mathrm{in}\  (0,\infty)  \times \mathbb{R}^d   =: \mathbb{R}^{d+1}_+\\
	\label{eq:manProb2}
	\uu(0,\cdot ) &= \uu_0,  &&\mathrm{in}\ \mathbb{R}^d ,
\end{align}
where   $f=(f_1, \dots , f_d)$ is a smooth flux function. We assume that $\uu_0\in\Linf(\mathbb{R}^d)$ with compact support. A classic result states the existence of a unique entropy solution in $\Linf(\mathbb{R}^{d+1}_+)$, still with compact support~\cite{godlewski:1991hyp,Kruzkov1970FIRST-ORDER-QUA}.
Recall that a weak solution to~\eqref{eq:manProb1} is an entropy solution if it satisfies, in the distributional sense,
\begin{equation}
	\entr(\uu)_t + \sum_{i=1}^d  \q_i(\uu)_ {x_i} \leq 0 
\end{equation}
for any convex entropy function $\entr(\uu)$ and  associated entropy flux function $\q=(\q_1, \dots, \q_d)$ that satisfy the compatibility condition $\q_i' = f_i'\entrp $ for $ i=1, \dots , d$.

We prove convergence of a class of DG schemes to the unique entropy solution of~\eqref{eq:manProb1},~\eqref{eq:manProb2}. Using DiPerna's theory of measure valued solutions~\cite{DiPerna1985Measure-valued-}, we can assert  strong convergence {in $\Lp{\altpoldeg}$ for $1\leq \altpoldeg < \infty$}, provided that the approximate solution is
\begin{enumerate}[i)]
\item  uniformly bounded in $\Linf(\mathbb{R}_+^{d+1})$,
\item weakly consistent with all entropy inequalities,
\item strongly consistent with the initial data.
\end{enumerate}
These statements will be made more precise below. It is worth noting that, for Cauchy problems of the type~\eqref{eq:manProb1},~\eqref{eq:manProb2}, one may use Szepessy's extension~\cite{szepessy:existence} to DiPerna's theory, to replace the first criterion with the weaker requirement that the approximate solution be bounded in $\Linf(\mathbb{R}_+;\Lp{\altpoldeg}( \mathbb{R}^{d}))$. We comment on this distinction below.

Convergence results along these lines have been obtained previously for both continuous Galerkin (CG) and DG  finite-element methods~\cite{Szepessy1989Convergence-of-,jaffre:1995convDG,Henke2013LL-boundedness-}. The schemes that were considered in the above references, both CG and DG, have in common that they use streamline-diffusion stabilization in conjunction with a shock-capturing operator. At least in the case of DG methods, streamline-diffusion stabilization is not commonly found in practical implementations~\cite{Wang2013High-order-CFD-,Bassi1997High-Order-Accu,cockburn2001:dgConvection,hartmann:2002errorHyperbCons}. For linear hyperbolic problems, the  analysis of DG schemes has long ago made the transition from such schemes that use streamline-diffusion stabilization~\cite{Houston2000Stabilized-hp-F} to schemes that do not~\cite{Houston2002Discontinuous-h}. However, the techniques used for linear problems have no direct extension to the nonlinear case. To the best of our knowledge, a proof of convergence, similar in scope to that presented here,  is not available for DG schemes  applied to nonlinear conservation laws, unless  streamline-diffusion stabilization is used. (In addition to the references cited above, we also mention the error estimates provided in~\cite{Cockburn1996Error-Estimates}.) 

A proof of convergence exists for CG schemes without streamline-diffusion stabilization, using polynomials of degree $\poldeg = 1$~\cite{Nazarov2013Convergence-of-}. In the present paper, convergence of a space-time DG scheme for~\eqref{eq:manProb1},~\eqref{eq:manProb2}, similar to that of~\cite{jaffre:1995convDG}, is proved for $\uu_0\in \Linf(\mathbb{R}^d)$ with compact support, and any degree $\poldeg\geq 1$ of polynomial approximation. The scheme does not use  streamline-diffusion stabilization, but only a suitable shock-capturing operator.
For our proof, we use DiPerna's framework \cite{DiPerna1985Measure-valued-}, requiring a uniform $\Linf$-bound on the approximate solution, as opposed to the modification due to Szepessy~\cite{szepessy:existence}. The reason is that, firstly, showing the existence of such a uniform bound is an interesting result in its own right. Secondly, the existence of such a bound  facilitates the proofs in this paper, and thirdly, it allows the extension to intial-boundary value problems. (See, for example, the results in~\cite{szepessy:1991convSDwBdy,Henke2013LL-boundedness-}.) 

The paper is organized as follows: In section~\ref{sec:spacetime} we introduce the space-time DG scheme, and define some notation. In section~\ref{sec:entropyForm}, using a suitable test function, we re-write the scheme in a convenient form  from which all the remaining results can be derived. We prove a uniform bound of the numerical solution in $\Linf$,  and the weak entropy consistency of our DG scheme in sections~\ref{sec:quadratic} and~\ref{sec:entropy}, respectively.  Noting consistency with initial data, we are then in a position to conclude that the DG scheme converges to the unique entropy solution.

\section{Space-Time DG Scheme}
\label{sec:spacetime}
%
%

We define the space-time coordinates $x=(x_0, x_1, \dots , x_d)$ with $x_0\equiv t$. It will sometimes  be convenient to write $\xb=(x_1, \dots , x_d)$ for  the spatial variables.
For arbitrary $T>0$, consider a sequence of time instances $0\equiv t_0 < t_1 < \dots < t_N \equiv T$, and define corresponding time intervals $I_n=(t_n,t_{n+1})$.  Let $\triagn=\{\telem\}$  be a subdivision of the space-time slab  $I_n \times \mathbb{R}^d$ into disjoint {simplices}. Define $\triag:= \bigcup_{n=0}^{N-1} \triagn$, and let  $\h:=\sup_{\telem\in\triag} \h_\telem$, where  $\h_\telem:=diam(\telem)$. (We may assume $\h < \infty$.)  Convergence analysis will be carried out for shape-regular families of triangulations $\{\triag\}$, which means there is a constant $C>0$, independent of $\h$, such that
\begin{equation}
	\label{eq:shapereg2}
	\frac{1}{C}\leq \frac{\h_\telem |\partial \telem|}{|\telem|} \leq C  , \quad \forall \telem \in \triag.
\end{equation}
 
We define the approximation spaces
\begin{equation}
	\label{eq:ApproxSpace}
	\Vspn = \left\{  \vv :\,  \vv|_{\telem}\in \polp(\telem)\ \forall\telem \in \triagn \right\}, \qquad \Vsp = \prod_{n=0}^{N-1} \Vspn ,
\end{equation}
where $\polp$ is the space of $d+1$-variate polynomials of total degree $\poldeg$. In the following, we assume that $\poldeg\geq 1$. (The case $\poldeg = 0$ is dealt with in Ref.~\cite{cockburn:errorFV1994}.) Owing to compact support of the initial data, and the finite domain of dependence, we can assume that $\vv=0$ for large $|x|$. 

The scheme which we analyze is defined by determining  $\uh\in \Vsp$, such that 
\begin{equation}
	\label{eq:DGfinalScheme}
	\semilin{\uh}{\vv} =  \semilinDG {\uh}{\vv}  + \semilinS{\uh}{\vv}  = 0, \qquad \forall \vv \in \Vsp.
\end{equation}
Introducing the shorthand notation $\sumnk := \sum_{n=0}^{N-1} \sum_{\telem\in\triagn}$, we  define the semi linear form
\begin{equation}
	\label{eq:DGspacetime}
		 \semilinDG{\uh}{\vv}  :=	\sumnk\left\{ \intk  \nabla \cdot \fst (\uh)\vv  \dx +  \intdk \left( \fhat - \fst(\uhk) \cdot \no\right)\vv\ds  \right\},
\end{equation}
where $\fst=(\uu,f(\uu))$ is the space-time flux. 
%
%
On $\partial\telem$ we  define the interior  trace values as $\uhk(x)=\lim_{\epsilon\to 0^-} \uh(x+\epsilon \no)$, where $\no=(\no_t, \no_{x_1},\dots , \no_{x_d})$ is the outward pointing normal on $\partial \telem$ at $x$. We denote as $\teleme\in \triag$ the element separated from $\telem$ on $e\subset \partial\telem$, and introduce the numerical flux function $\fhat \equiv\fhat (\uhk,\uhkn;\no)$, defined to be  Lipschitz-continuous, strictly monotone, and consistent with $\fst$. 
On interfaces where $\no=(\pm1,0,\dots, 0)^T$ we assume that the numerical flux reduces to pure upwinding (i.e. in time). In this case, for any $0\leq n \leq N-1$,  we can re-write the boundary integral 
\begin{equation}
	\label{eq:boundaryIntSplit}
	\sum_{\telem\in\triagn}  \intdk \! \left( \fhat - \fst(\uhk) \cdot \no\right)\vv\ds   =\sum_{\telem\in\triagn}  \intdkx \! \left( \fhat - \fst(\uhk) \cdot \no\right)\vv\ds  +\int_{\mathbb{R}_n^d} (\uhp - \uhm) \vvp\dxb ,
\end{equation}
where $\Rn:=\{ t_n\} \times\mathbb{R}^d $, and $\belemx$ is the "internal" boundary of  $\telem\in\triagn$, consisting of points that are shared by at least two elements in $I_n\times \mathbb{R}^d$. Furthermore for $x\in \Rn$, we  define $\vv^\pm:= \lim_{\epsilon\to 0}\vv(t_n\pm \epsilon,\xb)$.
 

We augment the semi linear form~\eqref{eq:DGspacetime} with a shock-capturing operator
\begin{equation}
	\label{eq:DGspacetimeS}
		 \semilinS{\uh}{\vv}  := \sumnk  \intk \epsvisc(\uh) \nabla \uh \cdot \nabla \vv\dx ,
\end{equation}
%
%
with the viscosity parameter
\begin{eqnarray}
	\label{eq:shockcapt}
	\epsvisc(\uh){|_{\telem}} :=  \h^\viscexp\, C_\epsvisc \frac{\intk |\nabla\cdot \fst(\uh)|\dx +  \intdk |\uhk - \uhkn| \ds}{\intk\dx}, \qquad \telem \in \triag,
\end{eqnarray}
where { $\z<\viscexp < 2$}, and $C_\epsvisc$ is a positive constant of order unity. (In the following  we set $C_\epsvisc=1$.) 


We define the norm
\begin{equation}
	\normgen{\fst'}_\infty :=  \sup_{\xi\in \mathbb{R}}|\fst'(\xi)| , \qquad | \fst' | := \left( \sum_{i=0}^d f_i'(\xi)^2 \right)^{\z},
\end{equation}
and assume that $\normgen{\fst'}_\infty < \infty$.\footnote{This is not a severe restriction, because the true solution to problem~\eqref{eq:manProb1},~\eqref{eq:manProb2} satisfies a maximum principle. In order to comply with the growth restriction, we may apply a smooth modification to $f(\vv)$  for $\vv$ outside the interval $[\min_{x\in \mathbb{R}^\dimx} \uu_0, \max_{x\in \mathbb{R}^\dimx} \uu_0]$.}   This implies that 
\begin{eqnarray}
	|\nabla\cdot \fst(\vv)| &\leq& \normgen{\fst'}_{\infty} |\nabla \vv|,\qquad \forall  \vv\in \Vsp .\label{eq:resEstm}
\end{eqnarray}
In the following $C$ shall denote a generic positive constant, independent of $\uh$ and $\h$, but not necessarily the same at each occurrence.

\section{The Canonical Entropy Test}
\label{sec:entropyForm}

All results in this paper can be obtained by considering in~\eqref{eq:DGfinalScheme}  test functions of the form $\vv=\projh(\entrp \vp)$. Here, $\projh$ is a suitable projection  operator into the space $\Vsp$,  $\vp$ is a smooth non-negative  function, and  $\entr$ is a convex entropy. 
Following~\cite{jaffre:1995convDG}, we define $\projh$ using an $\Hm{1}$-projection  such that for all $\telem\in \triag$
\begin{align}
	\label{eq:H1proj}
	\intk\nabla (\projh (\entrp \phi)) \cdot \nabla \vv \dx &= 	\intk\nabla (\entrp \phi) \cdot \nabla \vv \dx , \quad \forall \vv\in \polp(\telem)\\
	\intk \projh(\entrp \phi) \dx &= \intk \entrp \phi \dx. \notag
\end{align}
%
Using $\vv=\projh(\entrp \vp)$ in~\eqref{eq:DGfinalScheme} yields
\begin{equation}
	\label{eq:masterEqDG}
	 \semilinDG{\uh}{\entrp \vp}  +  \semilinDG{\uh}{\projh(\entrp \vp)-\entrp \vp}  +  \semilinS{\uh}{\entrp \vp}  = 0.
\end{equation}
Note that $ \semilinS{\uh}{\projh(\entrp \vp)}=  \semilinS{\uh}{\entrp \vp}$,  which is a consequence of the definition of the $\Hm{1}$-projection~\eqref{eq:H1proj}. 
We define the space-time entropy flux  $\qst:=(\entr, \q_1, \dots , \q_d)$. Then $\qst$ satisfies the extended compatibility relation
\begin{equation}
	\label{eq:spacetimeEntrRel}
	\qst_i'=\fst_i'\entrp, \qquad i=0, \dots , d.
\end{equation}
The first term in~\eqref{eq:masterEqDG} may be rewritten as
\begin{align}
	 \semilinDG{\uh}{\entrp\vp} &= \sumnk\left(\intk \nabla \cdot \qst(\uh)  \vp\dx + \intdkx \left( \fhat - \fst(\uhk) \cdot \no\right)\entrp(\uhk)\vp\ds\right) \notag \\
	&\qquad + \sum_{n=0}^{N-1} \int_{\Rn} (\uhp - \uhm) \entrp (\uhp) \vp\dxb,\notag%
\end{align}
where the boundary integrals are split as in~\eqref{eq:boundaryIntSplit}. Integration by parts yields
\begin{align}
	\label{eq:entropyFormTemp}
	&  \semilinDG{\uh}{\entrp \vp} = -\int_{\mathbb{R}^d_T} \qst(\uh)  \cdot \nabla\vp  \dx +\sumnk \intdkx \qst(\uhk)\cdot \no \vp\ds\\ &\qquad +\sumnk \intdkx  \left( \fhat - \fst(\uhk) \cdot \no\right)\entrp(\uhk)\vp\ds + \int_{\Rnn{N}}\entr(\uhm)\vp\dxb  \notag  \\ 
	& \qquad - \int_{\mathbb{R}^{d}}\entr(\uu_0)\vp\dxb+\sum_{n=0}^{N-1} \int_{\Rn} \left( \entr(\uhm)-\entr(\uhp) + (\uhp - \uhm) \entrp (\uhp)\right) \vp\dxb, \notag
\end{align}
where $\mathbb{R}^d_T:= (0,T)\times \mathbb{R}^d$ has been defined, and we have equated $\uhm\equiv \uu_0$ for $n=0$.  (In practice one might use a suitable projection for the initial conditions, but this is of minor importance here.)  Using~\eqref{eq:spacetimeEntrRel}, we have
\begin{align}
	\sumnk  \intdkx \qst(\uhk)\cdot \no \vp \ds &= \sumnk\z \intdkx( \qst(\uhk)- \qst(\uhkn))\cdot \no \vp \ds  \notag \\
	&=\sumnk\z \intdkx\int_{\uhkn}^{\uhk} \qst'(\xi)\cdot \no \vp \dxr \xi \ds\notag\\
	&=\sumnk\z \intdkx \left\{ \left. \fst\cdot\no \entrp \right|^{\uhk}_{\uhkn}  -\int_{\uhkn}^{\uhk} \fst(\xi)\cdot\no \entrpp(\xi)   \dxr \xi\right\} \vp \ds. \label{eq:tempIntByPart}
\end{align}
After rewriting the second surface integral over $\belemx$ in~\eqref{eq:entropyFormTemp} to include contributions from both sides of the interface, and noting
\begin{equation}
	\fhat (\uhk,\uhkn; \no ) ( \entrp(\uhk) -\entrp(\uhkn))\vp =   \int_{\uhkn}^{\uhk} (\fhat (\uhk,\uhkn; \no )\entrpp(\xi)\vp \dxr \xi , \notag
\end{equation}
we can substitute expression~\eqref{eq:tempIntByPart} into~\eqref{eq:entropyFormTemp} to  obtain
\begin{align}
	\label{eq:DGentrTestFinal}
	& \semilinDG{\uh}{\entrp \vp} =  -\int_{\mathbb{R}_T^d} \qst(\uh)  \cdot \nabla\vp  \dx + \int_{\Rnn{N}}\entr(\uhm)\vp\dxb- \int_{\mathbb{R}^{d}}\entr(\uu_0)\vp\dxb   \\ 
	& \qquad\qquad  +\sum_{n=0}^{N-1} \int_{\Rn} \left( \entr(\uhm)-\entr(\uhp) + (\uhp - \uhm) \entrp (\uhp)\right) \vp\dxb \notag\\
	& \qquad\qquad+ \sumnk \z \intdkx \int_{\uhk}^{\uhkn} \left( \fst(\xi)\cdot \no  -\fhat  \right) \entrpp(\xi)  \vp   \dxr \xi \ds.\notag
\end{align}
For later use, we note that, for any convex entropy, and non-negative  function $\vp$,
\begin{equation}
	\label{eq:nonnegativeEntr1}
	 \left( \entr(\uhm)-\entr(\uhp) + (\uhp - \uhm) \entrp (\uhp)\right) \vp \geq 0.
\end{equation}
Furthermore, by assumption, $\fhat$ is strictly monotone, which implies the E-flux property in the sense of Osher~\cite{osher:1984eScheme}, and so 
\begin{equation}
	\label{eq:nonnegativeEntr2}
	 \int_{\uhk}^{\uhkn} \left( \fst(\xi)\cdot \no  -\fhat (\uhk,\uhkn;\no) \right) \entrpp(\xi)  \vp  \dxr \xi \geq 0 .
\end{equation}

\section{Stability}
\label{sec:quadratic}


First, we  state a discrete  $\Ltwo$-stability result. Use $\entr= \frac{1}{2} \uh^2$ and $\vp \equiv1$ in~\eqref{eq:masterEqDG}, and observe that for this choice $\entrp(\uh)\in \Vsp$. Consequently the second term in~\eqref{eq:masterEqDG} vanishes. From~\eqref{eq:DGspacetimeS} and~\eqref{eq:DGentrTestFinal}, we obtain 
%
\begin{eqnarray}
			&&\sumnk \left\{   \intk  \epsvisc(\uh)\left| \nabla \uh\right|^2\dx +   \z\intdkx\int_{\uhk}^{\uhkn} (  \fst(\xi) \cdot \no-\fhat (\uhk,\uhkn; \no)) d\xi \, \ds\right\}  \notag \\ &&\qquad\qquad\qquad\qquad +\sum_{n=0}^{N-1} \z\int_{\Rn} (\uhp -\uhm)^2 \dxb  + \z\int_{\Rnn{N}} (\uhm)^2 \dxb= \z\int_{\mathbb{R}^\dimx} \uu^2_0 \dxb.\label{eq:stabResult}
\end{eqnarray}
Therefore, noting~\eqref{eq:nonnegativeEntr2}, clearly we have discrete $\Ltwo$-stability in the sense
\begin{equation}
	\normgen{\uh(\cdot , t_N^-)}_{\Ltwo(\mathbb{R}^{\dimx})}  \leq  \normgen{\uu_0}_{\Ltwo(\mathbb{R}^{\dimx})}  , \qquad N=1,2,\dots.
\end{equation}
At the interfaces $\Rn$, equation~\eqref{eq:stabResult} explicitly states an $\Ltwo$-bound on the jumps of the solution. For a strictly monotone numerical flux, a bound on the interface jumps of the solution $\uh$  at $\belemx$ is also implied~\cite{Cockburn1996Error-Estimates}, i.e. there exists a constant $C> 0$ such that 
\begin{equation}
	\label{eq:jumpBound}
	\intdkx\int_{\uhk}^{\uhkn} (  \fst(\xi) \cdot \no-\fhat (\uhk,\uhkn; \no)) \dxr \xi \, \ds \geq  C \intdkx ( \uhk -\uhkn)^2\ds. 
\end{equation}

Now take $\entr= \frac{1}{ \altpoldeg} \uh^\altpoldeg$, with an \emph{even} integer $\altpoldeg\geq 4$, and $\phi \equiv 1$ in~\eqref{eq:masterEqDG}. Using~\eqref{eq:DGentrTestFinal}, and the inequalities~\eqref{eq:nonnegativeEntr1},~\eqref{eq:nonnegativeEntr2}, we obtain immediately
%
\begin{align}
	\label{eq:LinfDiscrStart}
	&\frac{1}{\altpoldeg} \normgen{\uh(t^-_N,\cdot)}_{\Lp{\altpoldeg}(\mathbb{R}^\dimx)}^\altpoldeg +  \semilinS{\uh}{ \uh^{\altpoldeg-1}}+ \semilinDG {\uh}{ \projh (\uh^{\altpoldeg-1}) -\uh^{\altpoldeg-1}} \leq \frac{1}{\altpoldeg}  \normgen{\uu_0}_{\Lp{\altpoldeg}(\mathbb{R}^\dimx)}^\altpoldeg  . 
\end{align}
 %
The term involving   the projection error can be written
\begin{align}
	\semilinDG{\uh}{\projh (\uh^{\altpoldeg-1}) -\uh^{\altpoldeg-1}} &= \sumnk\left\{ \intk \nabla \cdot \fst (\uh) (\projh (\uh^{\altpoldeg-1}) -\uh^{\altpoldeg-1})\dx\right. \notag\\
	& \quad\quad + \left. \int_{\partial\telem} (\fhat -\fst(\uhk)\cdot\no ) (\projh ((\uhk)^{\altpoldeg-1}) -(\uhk)^{\altpoldeg-1}) \ds \right\} . \notag
\end{align}
The $\Hm{1}$-projection~\eqref{eq:H1proj} is a simple Neumann problem. For a smooth enough function $\vv$, we have the standard error bound $\normgen{\vv-\projh\vv}_{\Linf(\omega)}\leq C \h^{\poldeg+1}\normgen{\vv}_{\Wmp{\poldeg+1}{\infty}(\telem)}$ for $\omega=\telem$ or  $\omega=\partial\telem$. Setting $\vv=w^{\altpoldeg-1}$, where $w\in \Vsp$,  we can further rewrite the projection error  (cf.~\cite{szepessy:1991convSDwBdy})
\begin{equation}
	\label{eq:errorInfInt}
		\normgen{\projh (w^{\altpoldeg-1}) -w^{\altpoldeg-1}}_{\Linf(\telem)} \leq C  \h^{2}  \altpoldeg^{\poldeg+1} \normgen{w}^{p-3}_{\Linf(\telem)}\normgen{\nabla w}^2_{\Linf(\telem)}. 
\end{equation}
%

%
%
Using Lipschitz-continuity of the numerical flux, and the definition of the artificial viscosity~\eqref{eq:shockcapt}, we thus obtain
\begin{equation}
		|\semilinDG{\uh}{\projh (\uh^{\altpoldeg-1}) -\uh^{\altpoldeg-1}} |\leq C  \h^{2-\beta}  \altpoldeg^{\poldeg+1}  \sumnk \intk \epsvisc(\uh)\normgen{\uh}^{p-3}_{\Linf(\telem)}\normgen{\nabla \uh}^2_{\Linf(\telem)} \dx  . \notag
\end{equation}
Split the triangulation into  $\triagnn{>}:=\{\telem\in \triag : \, \normgen{\uh}_{\Linf(\telem)}>1\}$ and $\triagnn{<}:=\{\telem\in \triag : \, \normgen{\uh}_{\Linf(\telem)}<1\}$ to obtain
\begin{eqnarray}
	&& |\semilinDG{\uh}{\projh (\uh^{\altpoldeg-1}) -\uh^{\altpoldeg-1}} | \leq  \notag \\&& \quad C  \h^{2-\beta}   \altpoldeg^{\poldeg+1} \left\{ \sum_{\telem\in\triagnn{>}} \intk \epsvisc(\uh)\normgen{\uh}^{p-2}_{\Linf(\telem)}\normgen{\nabla \uh}^2_{\Linf(\telem)} \dx 
	+  \sum_{\telem\in\triagnn{<}} \intk \epsvisc(\uh) |\nabla \uh|^2 \dx  \right\}\notag
\end{eqnarray} 
where the standard estimate
\begin{equation}
	\label{eq:LinfLtwoInverse}
	|\telem| \normgen{\nabla \uh}^2_{\Linf(\telem)} \leq  \normgen{\nabla \uh}^2_{\Ltwo(\telem)}
\end{equation}
has been used. To proceed, we need the following lemma: 
\begin{lem} \label{Lem::SC_Coerc}
For a shape regular triangulation  $ \mathcal{T}_h = \{ \kappa \}	$, there is a $\h$-independent constant $\hat{C}(q, p)$ such that $\forall \kappa \in  \mathcal{T}_h$ 
\begin{equation}
	\intk \normgen{\vv}^{p-2}_{\Linf(\telem)} |\nabla \vv|^2\dx  \leq \hat{C}(q, p) \intk \nabla\vv \cdot \nabla (\projh(\vv^{\altpoldeg-1})) \dx, \qquad v\in \polp(\telem), \altpoldeg = 2, 4, \cdots . 
\end{equation}
\end{lem}
The proof is given in the appendix. 
This lemma was proved in~\cite{Szepessy1989Convergence-of-} for linear elements. In~\cite{Henke2013LL-boundedness-} the authors give a proof for $\vv\in \polp(\telem)$ and $\projh$ a  Lagrangian interpolation operator. In the appendix we present a novel proof for the lemma, which is simpler and applicable to the $\Hm{1}$-projection~\eqref{eq:H1proj}.
Using Lemma \ref{Lem::SC_Coerc},  $\Ltwo$-stability~\eqref{eq:stabResult}, and ~\eqref{eq:LinfLtwoInverse}, there follows
\begin{align}
	 |\semilinDG{\uh}{\projh (\uh^{\altpoldeg-1}) -\uh^{\altpoldeg-1}} | &  \leq  
	 C  \h^{2-\beta} \altpoldeg^{\poldeg+1}
	  \nonumber \\
	 & \quad+  \hat{C}(q, p)  \h^{2-\beta} \altpoldeg^{\poldeg+1} \sumnk \intk\eps(\uh) \nabla \uh\cdot \nabla (\projh(\uh^{\altpoldeg-1}))\dx.
\end{align}
Finally, substituting this result into~\eqref{eq:LinfDiscrStart}, we have
\begin{equation}
	\normgen{\uh(t^-_N,\cdot)}^\altpoldeg_{\Lp{\altpoldeg}(\mathbb{R}^\dimx)} + \altpoldeg (1 -\hat{C}(q, p) \h^{2-\beta} \altpoldeg^{\poldeg+1})\semilinS{\uh}{ \projh(\uh^{\altpoldeg-1})} \leq \normgen{\uu_0}^\altpoldeg_{\Lp{\altpoldeg}(\mathbb{R}^\dimx)}  + C \h^{2-\beta}\altpoldeg^{\poldeg+2}.\notag
\end{equation}
Clearly, in order to keep the $\Lp{\altpoldeg}$ norm bounded or all $t_n$,  $n=1,\dots, N$, it is  required that
\begin{align}
 C h^{2-\beta} p^{q+2} &\leq  \infty,\notag\\
 p - \hat{C}(q, p) h^{2-\beta} p^{q+2} &\geq 0. \notag
\end{align}

From these conditions we get
\begin{equation}
h^{2-\beta} \lesssim  \min \{  \dfrac{1}{p^{q+2}}, \dfrac{1}{ \hat{C}(q, p) p^{q+1} } \},\label{eq:hconstraint}
\end{equation}
and note that as $\h \to 0$ this will still permit us  to let $\altpoldeg \to \infty$ eventually.

It is worth pointing out that so far the presence  of streamline-diffusion stabilization (or lack thereof) is clearly not essential. In previous proofs~\cite{Szepessy1989Convergence-of-,szepessy:1991convSDwBdy,Henke2013LL-boundedness-}, the control over the residual provided by the streamline-diffusion stabilization becomes significant  when extending the discrete stability to a uniform bound. We now establish a uniform bound in $\Linf$ without using streamline-diffusion stabilization:

\begin{prop}
Let $\uh$ be the solution produced by scheme~\eqref{eq:DGfinalScheme}. There exists a constant $C>0$ such that
\begin{equation}
	\normgen{\uh}_{\Linf([0,T]\times \mathbb{R}^\dimx)} \leq C 
\end{equation}
\end{prop}

\begin{proof}
Fix $t \in [t_{n}, t_{n+1}]$ for some $n=0,\dots , N-1$. For $\telem\in\triagnn{n}$ we define 
  $\telem^>\subset \telem$ as the subset of $\telem$ with first coordinate $x_0\equiv t\in [t_n,t]$, and analogously we define $\belemx^>\subset \belemx$. 

Using the divergence theorem the following identity is easily established:
\begin{align}
	\label{eq:divergenceLp}
	\int_{\mathbb{R}^\dimx} \entr(\uh( t,\cdot)) \dxb&= \int_{\mathbb{R}^\dimx}\entr( \uh( t_{n}^+,\cdot))\dxb \\ &\quad +\sum_{\telem\in\triagnn{n}} \int_{\telem^>}  \nabla\cdot \qst (\uh) \dx   
	   -\sum_{\telem\in\triagnn{n}} \int_{\belemx^>} \qst(\uhk)\cdot \no\ds . \notag 
 \end{align}
We consider the  entropy $\eta  = \frac{1}{\altpoldeg} \uh^\altpoldeg$ for \emph{even} $\altpoldeg\geq 4$. If we re-write the surface integral over $\belemx$ to include contributions from both sides of the interface, the last term on the right hand side is equal to
\begin{eqnarray}
	\frac{1}{2} \sum_{\telem\in\triagnn{n}} \int_{\belemx^>} (\qst(\uhkn)-\qst(\uhk))\cdot \no\ds &=& \frac{1}{2} \sum_{\telem\in\triagnn{n}} \int_{\belemx^>}\int_{\uhk}^{\uhkn} \qst'(\xi)\cdot \no d\xi\ds . \notag
\end{eqnarray}
Using the pointwise estimate
\begin{equation}
	\qst' \cdot \no \leq \left( \sum_{i=0}^d \fst'_i \right)^{\frac{1}{2}} \left(   \sum_{i=0}^d(\entrp \no_i)^2\right)^{\frac{1}{2}} \leq \normgen{\fst'}_\infty |\entrp| , \notag
\end{equation}
we can bound 
\begin{equation}
	\int_{\uhk}^{\uhkn} \qst'(\xi)\cdot \no d\xi \leq C\int_{\uhk}^{\uhkn}  \left| \entrp(\xi)\right| d\xi \leq   \frac{C}{\altpoldeg}\left\{( \uhkn)^\altpoldeg +( \uhk)^\altpoldeg  \right\}. 
\end{equation}
Consequently, re-distributing the surface integrals to include contributions only from the interior face, one has
\begin{eqnarray}
	\frac{1}{2} \sum_{\telem\in\triagnn{n}} \int_{\belemx^>} (\qst(\uhkn)-\qst(\uhk))\cdot \no\ds  	&\leq &   \frac{C}{\altpoldeg}   \sum_{\telem\in\triagnn{n}}  \int_{\belemx^>} ( \uhk)^\altpoldeg  \ds \nonumber\\  &\leq &\frac{C}{\h} \frac{1}{\altpoldeg}  \sum_{\telem\in\triagnn{n}}\int_{\telem^>}\uh^\altpoldeg dx, \label{eq:LpestimateSurfTerm} 
\end{eqnarray}
where a trace inverse inequality  has been used.

We now consider the divergence term in~\eqref{eq:divergenceLp}. To that end, we first note the pointwise estimate
\begin{equation}
	\nabla\cdot \qst \leq \left ( \sum_{i=0}^d (\fst'_i)^2 \right)^{\frac{1}{2}} \left (\sum_{i=0}^d ( \uh^{\altpoldeg-1}\partial_{x_i} \uh)^2 \right)^{\frac{1}{2}} \leq  \frac{1}{\altpoldeg} \normgen{\fst'}_\infty  |\nabla (\uh^\altpoldeg)| , \notag
\end{equation}
whence
\begin{eqnarray}
	\sum_{\telem\in\triagnn{n}} \int_{\telem^>}  \nabla\cdot \qst (\uh) \dx 
	&\leq & \frac{C}{\altpoldeg}  \sum_{\telem\in\triagnn{n}} \intk \left| \nabla (\uh^{\altpoldeg})\right| \dx \notag \\ 
	&\leq & \frac{C}{\h\altpoldeg}  \sum_{\telem\in\triagnn{n}} \intk  \uh^{\altpoldeg} \dx. \label{eq:LpestimateDivTerm} 
\end{eqnarray}
Substituting~\eqref{eq:LpestimateSurfTerm} and~\eqref{eq:LpestimateDivTerm} into~\eqref{eq:divergenceLp}, we see that for some positive constants $C_1$, $C_2$, we have
\begin{equation}
	 \int_{\mathbb{R}^d} {\uh(t,\xb)}^\altpoldeg\dxb \leq\frac{C_1}{\h}\int_{t_n}^t \int_{\mathbb{R}^d} \uh(\tau,\xb)^\altpoldeg \dx d\tau + C_2. \notag
\end{equation}
Since we already have discrete stability $\normgen{\uh(t_n^-,\cdot)}_{\Lp{\altpoldeg}(\mathbb{R}^d)} < \infty$ for $t_n$ for $n=0,1,\dots, N$, a simple Gronwall argument allows us to conclude
\begin{equation}
	\sup_{0\leq t\leq t_N}\normgen{\uh(t,\cdot)}_{\Lp{\altpoldeg}(\mathbb{R}^d)} \leq C, \notag
\end{equation}
%
for all $\h$ and $\altpoldeg$  satisfying~\eqref{eq:hconstraint}. Using an inverse estimate, the rest of the proof can be completed as in~\cite{Szepessy1989Convergence-of-,Henke2013LL-boundedness-}. 
\end{proof}



\section{Entropy Consistency}
\label{sec:entropy}
%
%
%


We must  prove the weak consistency of scheme~\eqref{eq:DGfinalScheme} with all entropy inequalities. This is formally stated  in the following theorem:
\begin{thrm}
	\label{eq:entropyCons}
	Let $\uh$ be the solution produced by scheme~\eqref{eq:DGfinalScheme}. Then, for all convex entropies, and associated space-time entropy fluxes~\eqref{eq:spacetimeEntrRel}, there holds
	\begin{equation}
		\label{eq:entropyIneq}
		\liminf_{\h\to 0}\left\{  \int_{(0,T)\times\mathbb{R}^\dimx }\qst(\uh)\cdot\nabla \vp\dx\right\} \geq 0 ,\qquad \forall\vp\in\mathcal C^\infty_0((0,T)\times\mathbb{R}^d ;\mathbb{R^+}).
	\end{equation}
\end{thrm}

In the remainder of this section we prove Theorem~\ref{eq:entropyCons}. {It} is well known that it is sufficient to show~\eqref{eq:entropyIneq} for all Kru\v{z}kov entropy pairs, i.e. convex entropies $\entr_k=|\uu-k|$ with $k\in\mathbb{R}$, and associated entropy fluxes~\cite{Kruzkov1970FIRST-ORDER-QUA,DiPerna1985Measure-valued-}. Following~\cite{jaffre:1995convDG,DiPerna1985Measure-valued-} we define a regularized Kru\v{z}kov entropy function $\entr_{k,\delta}$, converging to $\entr_k$ as $\delta \to 0$, uniformly on compact sets, while maintaining  convexity and uniformly bounded first derivatives. This allows us to define an associated regularized entropy flux $\q_{k,\delta}$, which converges uniformly on compact sets to the Kru\v{z}kov entropy flux.  Such a regularization can be accomplished with standard mollifiers. A more detailed exposition can be found, for example, in~\cite{DiPerna1985Measure-valued-,jaffre:1995convDG,szepessy:existence}. For the remainder of this section we simply identify $\entr\equiv\entr_{k,\delta}$ and $\q\equiv\q_{k,\delta}$ for some fixed $\delta>0$, which allows us to assume differentiability and that all derivatives $\entr^{(q)}$ are bounded. After establishing~\eqref{eq:entropyIneq} for the regularized entropy pairs, we  let $\delta\to 0$, and conclude theorem~\ref{eq:entropyCons} for the Kru\v{z}kov entropy pairs by applying dominated convergence.

Before we proceed with the proof, we give a few preparatory lemmas. First, note that in~\eqref{eq:stabResult} there is no term controlling the $\Ltwo$-norm of the flux divergence. Such terms arise in a straightforward manner from the streamline-diffusion stabilization~\cite{jaffre:1995convDG,Cockburn1996Error-Estimates,Szepessy1989Convergence-of-}. The next  lemma shows that we can control the flux divergence even in the absence of the streamline-diffusion stabilization. We shall need this result to complete our proofs.
\begin{lem}
\label{lem:resControl2}
Let $\uh$ be the solution produced by scheme~\eqref{eq:DGfinalScheme}. Then there exists a  constant $C> 0 $, independent of $\h$ and $\uh$, such that
\begin{eqnarray}
	\label{eq:boundres}
	&&\sumnk \intk \h^{2\beta} |\nabla \cdot \fst (\uh)|^2 \dx <   C \h^\frac{4\beta}{3}
\end{eqnarray}
where $\viscexp$ is defined in~\eqref{eq:shockcapt}. 
\end{lem}

\begin{proof}
%
%
Split the integral in~\eqref{eq:boundres} into integrals over $\telem^<:= \telem \cap \{x:|\nabla \uh(x)| < \h^{-\viscexpc}\}$ and $\telem^>:= \telem \cap \{x:|\nabla \uh(x)| > \h^{-\viscexpc}\}$. 
Using~\eqref{eq:resEstm} we have 
\begin{eqnarray}
	\h^{2\beta} \sumnk \int_{\telem^<}  |\nabla\cdot \fst(\uh)|^2 \dx \leq C	\h^{2\beta}\sumnk   \int_{\telem^<}   |\nabla \uh|^2 \dx \leq C \h^{\frac{4\beta}{3}} , \notag
\end{eqnarray}
and
\begin{eqnarray}
	\h^{2\beta}\sumnk  \int_{\telem^>}  |\nabla \cdot \fst( \uh)|^2 \dx &\leq & C \h^{2\beta}\sum_{\telem\in\triag}  \int_{\telem^>}  |\nabla \cdot \fst( \uh)| |\nabla \uh| \dx \notag\\
	&\leq &C\h^{2\beta}\sumnk  \int_{\telem^>}  |\nabla\cdot \fst( \uh)| |\nabla \uh|^2 \h^\viscexpc  \dx\notag\\
	&\leq &C\h^{2\beta+\frac{\beta}{3}} \sumnk \normgen{\nabla \uh}^2_{\Linf(\telem)} \int_{\telem}  |\nabla\cdot \fst( \uh)|  \dx\notag\\
	&\leq &C\h^{2\beta-\frac{2\beta}{3}} \sumnk \normgen{\nabla \uh}^2_{\Linf(\telem)} \int_{\telem}  \epsvisc(\uh)  \dx\notag\\
	&\leq &C\h^{\frac{4\beta}{3}} \sumnk  \int_{\telem} \epsvisc(\uh) |\nabla \uh|^2 \dx.\notag
\end{eqnarray}
Here we have used  the definition of the artificial viscosity~\eqref{eq:shockcapt} in the second-but-last inequality, and again the standard inverse estimate~\eqref{eq:LinfLtwoInverse}  in the last inequality. The lemma now follows from $\Ltwo$-stability~\eqref{eq:stabResult}.
\end{proof}

%
\begin{lem}
	\label{eq:epsbound}
	Let $\uh$ be the solution produced by scheme~\eqref{eq:DGspacetime}, and let $\epsvisc(\uh)$ be defined as in~\eqref{eq:shockcapt}. Then there exists a constant $C\geq 0$ such that for small enough $\h$
	\begin{equation}
		\label{eq:lemmaEps}
		\sumnk\intk \epsvisc(\uh) \dx\leq C \h ^{\viscexp-\z} 
	\end{equation}
\end{lem}
\begin{proof}
	Using the definition~\eqref{eq:shockcapt}, we have 
	\begin{eqnarray}
		\sumnk\intk \epsvisc(\uh) \dx  &=&\sumnk\intk\h^\viscexp |\nabla \cdot \fst (\uh)| \dx + \sumnk\intdk\h^\viscexp |\uhk - \uhkn| \dx \notag\\
		&\leq & \left(\sumnk\intk  \h^{2\beta} |\nabla \cdot \fst (\uh)|^2\dx \right)^\z  \left(\sumnk\intk\dx \right)^\z \notag \\
		&& \qquad+ \h^{\viscexp-\z} \left(\sumnk\intdk\h  \dx \right)^\z \left(\sumnk\intdk (\uhk - \uhkn)^2 \dx\right)^\z. \notag
	\end{eqnarray}
	The lemma follows from shape-regularity~\eqref{eq:shapereg2}, Lemma~\ref{lem:resControl2}, and the $\Ltwo$-stability results~\eqref{eq:stabResult},~\eqref{eq:jumpBound}.  
\end{proof}

Recall that we have assumed $\viscexp > \frac{1}{2}$, so the left-hand side of~\eqref{eq:lemmaEps} tends to zero as $\h\to 0$. We are now ready to prove the main result. 

\begin{proof}{(of Theorem~\ref{eq:entropyCons})} 
We start by considering~\eqref{eq:masterEqDG} for the regularized Kruzkov entropy and general non-negative $\vp$ with compact support in $(0,T)\times \mathbb{R}^d$. Furthermore, 
using the $\Hm{1}$ projection~\eqref{eq:H1proj}, we obtain from~\eqref{eq:DGentrTestFinal},  noting~\eqref{eq:nonnegativeEntr1},~\eqref{eq:nonnegativeEntr2},
\begin{align}
	\int_{\mathbb{R}_T^d} \qst(\uh)  \cdot \nabla\vp  \dx& \geq  \semilinDG{\uh}{\projh(\entrp \vp)-\entrp \vp}  +  \semilinS{\uh}{\entrp \vp} .  \notag
\end{align}
%
Split the shock-capturing term as follows:
\begin{eqnarray}
	  \semilinS{\uh}{\entrp \vp}  &= & \left( \sumnk \intk\eps |\nabla\uh|^2 \entrpp(\uh) \vp\dx +  \sumnk \intk\eps\nabla\uh\cdot \nabla\vp \entrp (\uh)\dx\right). \notag
\end{eqnarray}
Clearly, the first term is non-negative, while the second term vanishes as $\h\to 0$, i.e.
\begin{equation}
	\sumnk \intk\eps\nabla\uh\cdot \nabla\vp \entrp (\uh)\dx  \leq  C\left( \sumnk \intk\eps|\nabla\uh|^2 \dx  \right)^\frac{1}{2}\left( \sumnk \intk\epsvisc (\uh) \dx\right)^\frac{1}{2} \notag
\end{equation}
vanishes by $\Ltwo$-stability~\eqref{eq:stabResult} and Lemma~\ref{eq:epsbound}. 
Now define  $\eentr:= \projh(\entrp\vp) - \entrp\vp$, and write
\begin{eqnarray}
	 \semilinDG{\uh}{\eentr} =	\sumnk \left\{\intk \nabla\cdot \fst(\uh)\eentr  \dx+ \intdk \left( \fhat - \fst(\uhk) \cdot n\right)\eentr\ds  \right\} \notag. 
\end{eqnarray}
Using the definition of the shock-capturing operator~\eqref{eq:shockcapt},  we finally obtain by the estimates provided in~\cite{jaffre:1995convDG} (see (3.8c) in \cite{jaffre:1995convDG}) 
\begin{eqnarray}
	|\semilinDG{\uh}{\eentr}| &\leq & C\sumnk  \h^{2-\viscexp}\intk \epsvisc(\uh)\left(\normgen{\nabla\uh }^2_{\Linf(\telem)}+ 1\right) \normgen{\vp}_{\Wmp{\poldeg+1}{\infty}(\telem)}(1+\normgen{\uh}^{\poldeg-1}_{\Linf(\telem)}) \dx \notag\\
	&\leq& C \sumnk   \h^{2-\viscexp}\intk \epsvisc(\uh)\left(|\nabla\uh|^2 + 1\right) \dx
\end{eqnarray}
where~\eqref{eq:LinfLtwoInverse} and the uniform $\Linf$-stability.  The remaining integrals are bound by $\Ltwo$-stability and Lemma~\ref{eq:epsbound}, and we can conclude that $|\semilinDG{\uh}{\eentr}| \to 0$, as $\h\to0$. 

We have thus proved Theorem~\ref{eq:entropyCons} for the regularized Kru\v{z}kov entropy pair $\entr_{\delta,k}$, $\q_{\delta,k}$. 
Using the uniform convergence of $\entr_{\delta,k}$  and  $\q_{\delta,k}$ as  
$\delta\to 0$ on compact sets to $\entr_k$ and  $\q_k$ respectively, the entropy inequality~\eqref{eq:entropyIneq} follows by dominated convergence.  
\end{proof}

%


\section{Conclusion}

We have considered a class of space-time DG schemes for scalar hyperbolic conservation laws which do not use any streamline diffusion stabilization. We have shown that a shock-capturing term, somewhat simplified compared to similar schemes~\cite{jaffre:1995convDG}, provides sufficient stabilization. For all orders of polynomial approximation, the scheme admits a bound on the solution in $\Linf({[0,T]}\times \mathbb{R}^\dimx)$, and  is weakly consistent with all entropy inequalities. 

Moreover, the scheme is  consistent with the initial data. The proof is omitted here, as it clearly does not depend on streamline-diffusion stabilization. See~\cite{jaffre:1995convDG} and the references cited therein for details on consistency with the initial data.

Using DiPerna's theorem \cite{DiPerna1985Measure-valued-} we can conclude that the scheme converges to the unique entropy solution of~\eqref{eq:manProb1} for all polynomial orders. Similar convergence proofs  for nonlinear conservation laws have thus far only been available for schemes that use streamline diffusion stabilization.

\section*{Acknowledgments}
The authors are supported by the Deutsche Forschungsgemeinschaft (German Research
Association) through grant GSC 111.
The second author  would like to thank Dr.  Lukas Doering from Lehrstuhl I f\"ur Mathematik in RWTH Aachen,  for many helpful discussions and suggestions on the appendix.
%
%
\appendix
\section{Proof of Lemma~\ref{Lem::SC_Coerc}}
\label{sec:appendix}
First, we  review  some properties and theorems in polynomial inequalities. More specifically,  we introduce the Kneser inequality and its extension to multivariate polynomials which is the main interest here. Then some properties in the equivalency of norms in finite dimensional spaces are revisited. Then we are ready to give a short proof of Lemma~\ref{Lem::SC_Coerc}.

Recall the classic Kneser inequality \cite[p. 260]{borwein1995polynomials}:
\begin{thrm}\label{Thm::KneserOrig}
Suppose $\f = \g \hh$ where $\g \in \mathcal{P}^m_c$ and $\hh \in \mathcal{P}^{n-m}_c$, where $\mathcal{P}^m_c$ denotes the polynomials of order $m$ of one complex variable and complex coefficients. Then
\begin{equation}\label{Eq::KneserOrig}
\Vert \g \Vert_{[-1, 1]} \Vert \hh \Vert_{[-1, 1]} \leq \dfrac{1}{2} C_{n,m} C_{n, n-m} \Vert \f \Vert_{[-1, 1]},
\end{equation}
where by $\Vert \cdot \Vert_{[-1, 1]}$ we mean the uniform norm on the interval $[-1 , 1]$ of real numbers for continuous functions, and the coefficient $C_{n, m}$ is defined as
\begin{equation}
C_{n, m} := 2^m \prod_{k=1}^{m}  \big(1+ \cos \frac{(2k-1) \pi}{2n} \big).\notag
\end{equation}
\end{thrm}
One might simplify the complicated form of the coefficients $C_{n, m}$ with this corollary from \cite[p. 261]{borwein1995polynomials}:
\begin{cor} Under the assumptions of Theorem \ref{Thm::KneserOrig}, the following holds:
\begin{equation}\label{Eq::KneserSimpli}
\Big(\dfrac{\Vert \g \Vert_{[-1, 1]} \Vert \hh \Vert_{[-1, 1]}}{\Vert \f \Vert_{[-1, 1]}} \Big)^{1/n} \lesssim 3.3.
\end{equation}
\end{cor}

By a change of variables, inequality \eqref{Eq::KneserOrig} can be easily extended to any real interval $J = [a, b]$, i.e.
\begin{equation}\label{Eq::KneserAny}
\Vert \g \Vert_J \Vert \hh \Vert_J \leq K_n \Vert \f \Vert_J.
\end{equation} 
Let us call $K_n$ \emph{Kneser's constant}, which only depends on $n$. Using  \eqref{Eq::KneserSimpli}, we can  set  $K_n=3.3^n$ . Also note that by the definition of the uniform norm we may write $\Vert \cdot \Vert_J = \Vert \cdot \Vert_{\Linf (J)}$.

Here we are interested in the case where the polynomial coefficients are real, and  we seek the extension of Theorem~\ref{Thm::KneserOrig} to general multivariate real polynomials of total degree $n$
in a convex domain  $J \subset \mathbb{R}^d$. We can prove the following theorem, based on \cite[proof of Theorem 2.2]{andr2006norms}:
\begin{thrm} Assume $J\subset \mathbb{R}^d$ is convex and $\f$ is a polynomial of $d$ variables of total degree $n$ defined on $J$ such that $\f = \g\hh$ where the factors $\g$ and $\hh$ are polynomials of total degree less than $n$. Then the following holds
\begin{equation}\label{Eq::Knesermulti}
\Vert \g \Vert_{\Linf(J)} \Vert \hh \Vert_{\Linf(J)} \leq K_n \Vert \f \Vert_{\Linf(J)},
\end{equation}
where $K_n$ is the same as the constants defined in \eqref{Eq::KneserOrig} or  \eqref{Eq::KneserSimpli}.
\end{thrm}
\begin{proof}
Let us assume the $x_0, x_1 \in J$ are the maximum points of $\g$ and $\hh$ respectively, i.e.\
\begin{equation}
\Vert \g \Vert_{\Linf(J)}  = \vert \g(x_0) \vert, \qquad \Vert \hh \Vert_{\Linf(J)}  = \vert \hh(x_1) \vert.\notag
\end{equation}
Define the line segment $S:=\{x=t x_1 + (t-1) x_0; t\in [0,1]\}$. Due to convexity of $J$, one knows that  $S \subset J$. So the Kneser inequality holds on this line
\begin{equation}
\Vert \g \Vert_{\Linf(S)} \Vert \hh \Vert_{\Linf(S)} \leq K_n \Vert \f \Vert_{\Linf(S)}.\notag
\end{equation}
The special property of segment $S$, $x_0, x_1 \in S$, gives
\begin{equation}
\Vert \g \Vert_{\Linf(J)}  =  \Vert \g \Vert_{\Linf(S)} , \qquad  \Vert \hh \Vert_{\Linf(J)}  =  \Vert \hh \Vert_{\Linf(S)},\notag
\end{equation}
On the other hand it is clear that $\Vert \f \Vert_{\Linf(S)} \leq \Vert \f \Vert_{\Linf(J)}$, so \eqref{Eq::Knesermulti} holds.
\end{proof}

Also we will use the following equivalency of $\Lp{p}$ norms for a general function $f$ in a finite ($N$-) dimensional {space} $\mathbb{X}_N$ and $1 \leq q \leq p \leq \infty$,
\begin{equation}\label{Eq::Nrmequi}
\dfrac{1}{\vert J \vert^{1/q - 1/p} }\Vert  f \Vert_{\Lp{q}(J)} \leq  \Vert f \Vert_{\Lp{p}(J) }  \leq \dfrac{\hat{C}(\mathbb{X}_N, p, q, d)}{\vert J \vert^{1/q - 1/p}} \Vert f \Vert_{\Lp{q}(J)},
\end{equation}
where the left inequality comes directly from H\"older's inequality and the right inequality can be found in \cite[p. 140]{ciarlet1978finite}. Note that by the dependency on $\mathbb{X}_N$ we mean the dependency on the type of this space like $d$ and $N$ (or, in our case, equivalently the polynomial total degree $n$).

Now we are ready to present the proof of Lemma \ref{Lem::SC_Coerc}.

\begin{proof}[Proof of Lemma \ref{Lem::SC_Coerc}] Let us consider $J = \telem$ in the result showed before. By definition of the $\Hm{1}$ projection~\eqref{eq:H1proj} one has
\begin{equation}
\intk \nabla\vv \cdot \nabla (\projh(\vv^{\altpoldeg-1})) \dx = \intk \nabla\vv \cdot \nabla (\vv^{\altpoldeg-1})  \dx = (p-1) \intk \vert \nabla\vv \vert^2 \vv^{\altpoldeg-2} \dx.\notag
\end{equation}
The goal of the proof is to find a lower bound for the $\Vert \vert \nabla\vv \vert^2 \vv^{\altpoldeg-2}  \Vert_{\Lp{1}(\telem)}$ for even integer $p$ as the following
\begin{align}\label{Eq::result}
\lpnorm{ \vert \nabla\vv \vert^2   \vv^{\altpoldeg-2} }{1}{\telem}  \geq C \lpnorm{ \vert \nabla\vv \vert^2 }{1}{\telem}  \lpnorm{ \vv^{\altpoldeg-2} }{\infty}{\telem}  
 \end{align}	
Here $ \vert \nabla\vv \vert^2  $ and $\vv^{\altpoldeg-2}$ are polynomials; since $\vv \in \polp $, we know that 
\begin{equation*}
\nabla \vv \in (\polsp^{\poldeg-1})^d, \quad \vv^{p-2} \in \polsp^{\poldeg (p-2)},\quad \vert \nabla \vv\vert^2 \vv^{p-2} \in \polsp^{p \poldeg-2}.\notag
\end{equation*} 

Using the equivalence of norms ($\Lp{1}$ and $\Linf$) as in \eqref{Eq::Nrmequi} gives
\begin{equation}
\lpnorm{ \vert \nabla\vv \vert^2 \vv^{\altpoldeg-2}  }{1}{\telem} \geq \dfrac{\seminorm{J}} {\hat{C}(\polsp^{p \poldeg-2})}  \lpnorm{ \vert \nabla\vv \vert^2 \vv^{\altpoldeg-2}  }{\infty}{\telem}\notag
\end{equation}
Next we apply the Kneser inequality in the form of \eqref{Eq::Knesermulti} and again norm equivalency \eqref{Eq::Nrmequi} to get
\begin{align}
\lpnorm{ \vert \nabla\vv \vert^2  \vv^{\altpoldeg-2} }{1}{\telem}& \geq \dfrac{\seminorm{J}} { K_{pq-2} \hat{C}(\polsp^{p \poldeg-2})}  
  \lpnorm{ \vert \nabla\vv \vert^2  }{\infty}{\telem}  \lpnorm{   \vv^{\altpoldeg-2} }{\infty}{\telem}
  \nonumber \\
&\geq \dfrac{1}{K_{pq-2} \hat{C}(\polsp^{p \poldeg-2})}  \lpnorm{ \vert \nabla\vv \vert^2  }{1}{\telem}  \lpnorm{   \vv^{\altpoldeg-2} }{\infty}{\telem}.\notag
\end{align}
So the lower bound constant in \eqref{Eq::result} is $\dfrac{1}{K_{pq-2} \hat{C}(\polsp^{p \poldeg-2})}$; which yields
\begin{equation}
	\intk \normgen{\vv}^{p-2}_{\Linf(\telem)} |\nabla \vv|^2\dx  \leq \dfrac{K_{pq-2} \hat{C}(\polsp^{p \poldeg-2})}{p-1}  \intk \nabla\vv \cdot \nabla (\projh(\vv^{\altpoldeg-1})) \dx.\notag
\end{equation}
\end{proof}

For brevity we show the constant $\dfrac{K_{pq-2} \hat{C}(\polsp^{p \poldeg-2})}{p-1} $ by its two effective parameters as $\hat{C}(q, p)$. 
\bibliographystyle{siam}
\bibliography{bibtex_database}

\begin{thebibliography}{10}

\bibitem{Bassi1997High-Order-Accu}
{\sc F.~Bassi and S.~Rebay}, {\em High-order accurate discontinuous finite
  element solution of the 2d euler equations}, Journal of Computational
  Physics, 138 (1997), pp.~251--285.

\bibitem{borwein1995polynomials}
{\sc Peter Borwein}, {\em Polynomials and polynomial inequalities}, Springer
  Science \& Business Media, 1995.

\bibitem{ciarlet1978finite}
{\sc Philippe~G Ciarlet}, {\em The finite element method for elliptic
  problems}, Elsevier, 1978.

\bibitem{cockburn:errorFV1994}
{\sc Bernardo Cockburn, Fr\`ed\`eric Coquel, and Philippe LeFloch}, {\em An
  error estimate for finite volume methods for multidimensional conservation
  laws}, Math. Comput., 63 (1994), pp.~77--103.

\bibitem{Cockburn1996Error-Estimates}
{\sc Bernardo Cockburn and Pierre-Alain Gremaud}, {\em Error estimates for
  finite element methods for scalar conservation laws}, SIAM Journal on
  Numerical Analysis, 33 (1996), pp.~522--554.

\bibitem{cockburn2001:dgConvection}
{\sc B.~Cockburn and C.~W. Shu}, {\em {Runge-Kutta Discontinuous Galerkin}
  methods for convection-dominated problems}, J.~Sci.~Comp., 16 (2001),
  pp.~173--261.

\bibitem{DiPerna1985Measure-valued-}
{\sc Ronald~J. DiPerna}, {\em Measure-valued solutions to conservation laws},
  Arch. Ration. Mech. Anal., 88 (1985), pp.~223--270.

\bibitem{godlewski:1991hyp}
{\sc Edwige Godlewski and Pierre-Arnaud Raviart}, {\em Hyperbolic systems of
  conservation laws}, vol.~3/4 of Math\'ematiques \& Applications, Ellipses,
  1991.

\bibitem{hartmann:2002errorHyperbCons}
{\sc Ralf Hartmann and Paul Houston}, {\em Adaptive {Discontinuous Galerkin}
  finite element methods for nonlinear hyperbolic conservation laws}, SIAM J.
  Numer. Anal., 24 (2002), pp.~979--1004.

\bibitem{Henke2013LL-boundedness-}
{\sc Christian Henke and Lutz Angermann}, {\em
  ${L^\infty(L^\infty)}$-boundedness and convergence of {DG(p)} solutions for
  nonlinear conservation laws with boundary conditions}, IMA Journal of
  Numerical Analysis, 34 (2014), pp.~1598--1624.

\bibitem{Houston2000Stabilized-hp-F}
{\sc Paul Houston, Christoph Schwab, and Endre Suli}, {\em Stabilized hp-finite
  element methods for first-order hyperbolic problems}, SIAM Journal on
  Numerical Analysis, 37 (2000), pp.~1618--1643.

\bibitem{Houston2002Discontinuous-h}
\leavevmode\vrule height 2pt depth -1.6pt width 23pt, {\em Discontinuous
  hp-finite element methods for advection-diffusion-reaction problems}, SIAM
  Journal on Numerical Analysis, 39 (2002), pp.~2133--2163.

\bibitem{jaffre:1995convDG}
{\sc J.~Jaffre, C.~Johnson, and A.~Szepessy}, {\em Convergence of the
  {Discontinuous Galerkin} finite element method for hyperbolic conservation
  laws}, Math. Models Meth. Appl. Sci., 5 (1995), pp.~367--386.

\bibitem{andr2006norms}
{\sc Andras Kroo}, {\em On norms of factors of multivariate polynomials on
  convex bodies}, Electronic Transactions on Numerical Analysis, 25 (2006),
  pp.~201--205.

\bibitem{Kruzkov1970FIRST-ORDER-QUA}
{\sc S~N Kru\v{z}kov}, {\em First order quasilinear equations in several
  independent variables}, Mathematics of the USSR-Sbornik, 10 (1970), p.~217.

\bibitem{Nazarov2013Convergence-of-}
{\sc Murtazo Nazarov}, {\em Convergence of a residual based artificial
  viscosity finite element method}, Computers \& Mathematics with Applications,
  65 (2013), pp.~616--626.

\bibitem{osher:1984eScheme}
{\sc S.~Osher}, {\em Riemann solvers, the entropy condition, and difference
  approximations}, SIAM J. Numer. Anal., 21 (1984), pp.~217--235.

\bibitem{Szepessy1989Convergence-of-}
{\sc Anders Szepessy}, {\em Convergence of a shock-capturing streamline
  diffusion finite element method for a scalar conservation law in two space
  dimensions}, Mathematics of Computation, 53 (1989), pp.~527--545.

\bibitem{szepessy:existence}
{\sc Anders Szepessy}, {\em An existence result for scalar conservation laws
  using measure valued solutions}, Commun. Partial Differ. Equ., 14 (1989),
  pp.~1329--1350.

\bibitem{szepessy:1991convSDwBdy}
\leavevmode\vrule height 2pt depth -1.6pt width 23pt, {\em Convergence of a
  streamline diffusion finite element method for scalar conservation laws with
  boundary conditions}, M2AN, Math. Model. Numer. Anal., 25 (1991),
  pp.~749--782.

\bibitem{Wang2013High-order-CFD-}
{\sc Z.~J. Wang, Krzysztof Fidkowski, R{\'e}mi Abgrall, Francesco Bassi, Doru
  Caraeni, Andrew Cary, Herman Deconinck, Ralf Hartmann, Koen Hillewaert, H.~T.
  Huynh, Norbert Kroll, Georg May, Per-Olof Persson, Bram van Leer, and Miguel
  Visbal}, {\em High-order {CFD} methods: current status and perspective},
  International Journal for Numerical Methods in Fluids, 72 (2013),
  pp.~811--845.

\end{thebibliography}

\end{document}